\documentclass[12pt,reqno, twoside]{article}
\usepackage{amsfonts, amsmath, amssymb, amscd, amsthm}

\usepackage{multicol}
\usepackage{comment}

 \newtheorem{thm}{Theorem}[section]
 
 \newtheorem{lem}[thm]{Lemma}
 
 \theoremstyle{definition}
 
 \theoremstyle{remark}
 \newtheorem{rem}[thm]{Remark}
 \newtheorem{ex}{Example}
 
 \theoremstyle{claim}
 \numberwithin{equation}{section}

\numberwithin{equation}{section}

\newcounter{rom}
\renewcommand{\therom}{(\roman{rom})}
{\end{list}}

\title{A differentiable sphere theorem for compact Lagrangian submanifolds in complex Euclidean space and complex projective space}

\author{Haizhong
Li\thanks{Supported by Tsinghua University--K.U.Leuven Bilateral
scientific cooperation Fund.}
\thanks {Supported by NSFC grant No. 10971110} \and Xianfeng Wang\protect\footnotemark[1]
\thanks {Supported by NSFC grant No. 11171175}
}

\date{}

\begin{document}

\maketitle


\begin{abstract}\noindent We obtain a new differentiable sphere theorem for compact Lagrangian submanifolds
in complex Euclidean space and complex projective space.
\end{abstract}

{{\bfseries Key words}:  {\em differentiable sphere theorem,
Lagrangian submanifold, Ricci flow, mean curvature, second
fundamental form}.

{\bfseries Subject class: } 53C20, 53C40.}

\section{Introduction}
The sphere theorem for Riemannian manifolds was firstly studied by
Rauch(\cite{Ra}) in 1951, since then there have been many excellent
works on sphere theorems for Riemannian manifolds and
submanifolds(see
\cite{Andrews02}-\cite{Andrews2},\cite{Berger0}-\cite{Brendle4},\cite{Grove}-\cite{Kli},\cite{Micallef},\cite{Shiohama},\cite{Shiohama1}).

B. Andrews and C. Baker proved in \cite{Andrews} by the method of
mean curvature flow that, for a compact n-dimensional submanifold in
$\mathbb{R}^{n+p}$, we denote by $S$ the norm square of the second
fundamental form and $H$ the mean curvature, if
\begin{equation}\label{1.1}
S\le a n^2H^2, ~ \text{where} ~ a\le \frac{4}{3n}(2\le n\le 4);a\le
\frac{1}{n-1}(n\ge 4),
\end{equation}
then $M$ is diffeomorphic to $\mathbb{S}^n$.

S. Brendle and R. Schoen studied (\cite{Brendle2}-\cite{Brendle4})
the convergence theory for Ricci flow and its application to the
differentiable sphere theorem, they proved the following results,
which is very important in the proof of our main theorem for $n\ge
4$.

\begin{thm}\label{Lemma2.2}(see Theorem 2 in \cite{Brendle3})
Let $(M,g_0)$ be a compact, locally irreducible Riemannian manifold
of dimension $n(\geq4)$. Assume that $M\times \mathbb{R}^2$ has
nonnegative isotropic curvature, i.e.,
$$R_{1313}+\lambda^2 R_{1414} +\mu^2 R_{2323} + \lambda^2\mu^2R_{2424}-
2\lambda\mu R_{1234}\geq0$$ for all orthonormal four-frames
$\{e_1,e_2,e_3,e_4\}$ and all $\lambda,\mu\in[-1,1]$. Then one of
the following statements holds:

$(i)$ M is diffeomorphic to a spherical space form.

$(ii)$ $n=2m$ and the universal cover of M is a K\"{a}hler manifold
biholomorphic to $\mathbb{C}P^m$.

$(iii)$ The universal cover of M is isometric to a compact symmetric
space.
\end{thm}

In \cite{Xu-Gu2}, by applying Theorem \ref{Lemma2.2}, H. W. Xu and
J. R. Gu proved that (see also \cite{Xu-Gu})
\begin{thm}\label{thm1.02}(see Theorem 1.4 in \cite{Xu-Gu2}) Let $M$ be an $n(\geq4)$-dimensional
oriented complete submanifold in an $N$-dimensional simply connected
space form $F^{N}(c)$ with $c\ge0$. Assume that
$$S\le\frac{n^2H^2}{n-1}+2c,$$ where $c+H^2>0$. We have\\
$(i)$ If $c=0$, then M is either diffeomorphic to $S^n$, $\mathbb{R}^n$, or locally isometric to $S^{n-1}(r)\times \mathbb{R}$.\\
$(ii)$ If M is compact, then M is diffeomorphic to $S^n$ .
\end{thm}

Let $\bar{M}^n(4c)$ be a complex space form with constant
holomorphic sectional curvature $4c$, when $c=0$,
$\bar{M}^n(4c)=\mathbb{C}^n$; when $c>0$,
$\bar{M}^n(4c)=\mathbb{CP}^n$; when $c<0$,
$\bar{M}^n(4c)=\mathbb{CH}^n$.

For any Lagrangian submanifold $M$ in a complex space form
$\bar{M}^n(4c)$, the norm square of the second fundamental form and
the squared mean curvature satisfy the following inequality:
\begin{equation}\label{1.4}
S\geq \frac{3n^2H^2}{n+2},
\end{equation}
and the equality holds if and only if $M$ are totally geodesic
submanifolds or Whitney spheres in complex forms (see \cite{CU1},
\cite{BCM} and \cite{RU} for $c=0$; see \cite{Ch1996} and
\cite{Ch-V} for $c\neq0$).

The explicit expressions of the Whitney spheres in $\mathbb{C}^n$ or
in $\mathbb{CP}^n$ are the following examples.
\begin{ex}
\textit{Whitney sphere}  in $\mathbb{C}^n$ (see  \cite{BCM},
\cite{CU1}, \cite{RU}, \cite{LV}). It is defined as the Lagrangian
immersion of the unit sphere $\mathbb{S}^n$, centered at the origin
of $\mathbb{R}^{n+1}$, in $\mathbb{C}^n$, given by
\begin{equation}\phi:\mathbb{S}^n\to
\mathbb{C}^n:\phi(x_1,x_2,\ldots,x_n,x_{n+1})=\frac{1+ix_{n+1}}{1+x_{n+1}^2}(x_1,\ldots,x_n).\label{11}\end{equation}
\end{ex}

\begin{ex}
\textit{Whitney spheres} in $\mathbb{CP}^n$ (see \cite{Ch-V},
\cite{CMU2}, \cite{LV}). They are a one-parameter family of
Lagrangian spheres in $\mathbb{CP}^n$, given by
$$\bar{\phi}_{\theta}:\mathbb{S}^n\to
\mathbb{CP}^n(4):$$
\begin{equation}\bar{\phi}_{\theta}(x_1,x_2,\ldots,x_n,x_{n+1})=\pi\circ\Big(\frac{(x_1,\ldots,x_n)}{c_{\theta}+i
s_{\theta}x_{n+1}};\frac{s_{\theta}c_{\theta}(1+x_{n+1}^2)+ix_{n+1}}{c_{\theta}^2+s_{\theta}^2x_{n+1}^2}\Big),\label{12}
\end{equation} where $\theta\ge0$,
$c_{\theta}=\cosh \theta,~s_{\theta}=\sinh \theta$,
$\pi:\mathbb{S}^{2n+1}(1)\to \mathbb{CP}^n(4)$ is the Hopf
fibration.
\end{ex}

It is well-known that(see \cite{Smo11}), there are no self-shrinking
Lagrangian spheres in $\mathbb{C}^n$, if $n>1$. Our aim in this
paper is to prove a differentiable sphere theorem with weakly
pinched conditions
 for compact Lagrangian submanifolds in
$\mathbb{C}^n$ or in $\mathbb{CP}^n$, and the theorem is also valid
for Whitney spheres. In fact, we prove the following theorem:
\begin{thm}\label{main}
Let $M$ be an $n(\geq 3)$-dimensional compact Lagrangian submanifold
in a complex space form $\bar{M}^n(4c)(c\ge 0)$. We denote by $S$
the norm square of the second fundamental form and $H$ the mean
curvature. Assume that
\begin{equation}\label{pinching}
S\le\frac{3n^2H^2}{n+\frac{3}{2}}+2c,
\end{equation}
then $M$ is diffeomorphic to a spherical space form. In particular,
if $M$ is simply connected, then $M$ is diffeomorphic to
$\mathbb{S}^n$.
\end{thm}

\begin{rem}
When $c=0$, I. Castro (see \cite{C}) constructed a one-parameter
family of Lagrangian spheres including the Whitney spheres, defined
by
$$
\Phi_q(x_1,\ldots,x_n,x_{n+1})=\frac{2^{1/q}e^{i\beta_q(x_{n+1})}}{[(1+x_{n+1})^q+(1-x_{n+1})^q]^{1/q}}(x_1,\ldots,x_n),~q>1,
$$
with
$$
\beta_q(x_{n+1})=\frac{2}{q}\arctan\Big(\frac{(1+x_{n+1})^{q/2}-(1-x_{n+1})^{q/2}}{(1+x_{n+1})^{q/2}+(1-x_{n+1})^{q/2}}\Big),
$$
each $\Phi_q$ satisfies that
$S=\frac{(3n+q^2+2q-2)n^2H^2}{(n+q)^2}$.
We note that if
$$
1<q\le 2+\frac{3+\sqrt{3(2n^2+n-3)}}{2n-3},$$
then $\Phi_q$ satisfies our condition \eqref{pinching}, and we also note that $\Phi_2$ is Whitney sphere.
\end{rem}
\section{Preliminaries}
In this section, $M$ will always denote an n-dimensional
Lagrangian submanifold of $\bar{M}^n(4c)$ which is an n-dimensional
complex space form with constant holomorphic sectional curvature
$4c$. We denote the Levi-Civita connections on $M$,
$\bar{M}^n(4c)$ and the normal bundle by $\nabla$, $D$ and
$\nabla_X^{\bot}$, respectively. The formulas of Gauss and Weingarten
are given by (see \cite{CLU},\cite{DLVW},\cite{LV},\cite{LW})
\begin{equation}
D_XY=\nabla_XY+h(X,Y),
~D_X\xi=-A_{\xi}X+\nabla_X^{\bot}\xi,\label{21}
\end{equation}
where $X$ and $Y$ are tangent vector fields and $\xi$ is a normal
vector field on $M$.

The Lagrangian condition implies that

\begin{equation}\nabla_X^{\bot}JY=J\nabla_XY, ~A_{JX}Y=-J h(X,Y)=A_{JY}X,\label{22}\end{equation}
where $h$ is the second fundamental form and $A$ denotes the shape
operator.

The above formulas immediately imply that $\langle h(X,Y),JZ\rangle$ is totally symmetric, i.e.,
\begin{equation}\langle h(X,Y),JZ\rangle ~=~\langle h(X,Z),JY\rangle ,\label{29}\end{equation}
for tangent vector fields $X$, $Y$ and $Z$.

For a Lagrangian submanifold $M$ in $\bar{M}^n(4c)$, an orthonormal frame field
$$e_1,\ldots,e_n,e_{1^*},\ldots,e_{n^*}$$
is called an adapted Lagrangian frame field if $e_1,\ldots,e_n$ are orthonormal tangent vector fileds and $e_{1^*},\ldots,e_{n^*}$
are normal fields given by
\begin{equation}
e_{1^*}=Je_1,\ldots,e_{n^*}=Je_n.
\end{equation}
Their dual frame fields are $\theta_1,\ldots,\theta_n$, the Levi-Civita connection forms and normal connection forms are $\theta_{ij}$ and
$\theta_{i^*j^*}$, respectively.

Writing $h(e_i,e_j)=\sum\limits_{k=1}^nh_{ij}^{k^*}e_{k^*}$,
\eqref{29} is equivalent to
\begin{equation}\label{symm}
h_{ij}^{k^*}=h_{kj}^{i^*}=h_{ik}^{j^*},1\le i,j,k\le n.
\end{equation}

The norm square of the second fundamental form is
$S=\sum\limits_{i,j,k}(h_{ij}^{k^*})^2$. The mean curvature vector
$\vec{H}$ is defined by
$\vec{H}=\frac{1}{n}\sum\limits_{i,k}h_{ii}^{k^*}e_{k^*}$ and the
mean curvature $H=|\vec{H}|$.

 If we denote the components of
curvature tensor of $\nabla$  by $R_{ijkl}$, then the equations of Gauss
 are given by (see \cite{CLU},\cite{DLVW},\cite{LV},\cite{LW})
\begin{equation}\begin{aligned}R_{ijkl}=c(\delta_{ik}\delta_{jl}-\delta_{il}\delta_{jk})
+\sum_{r=1}^n(h_{ik}^{r^*}h_{jl}^{r^*}-h_{il}^{r^*}h_{jk}^{r^*}).\label{24}\end{aligned}\end{equation}

\section{Some Lemmas and the proof of Theorem \ref{main}}

We need the following lemmas to finish the proof of Theorem
\ref{main}.

In view of a result of Aubin, we have the following lemma:
\begin{lem}\label{Aubin}(see Aubin, \cite{Aubin})
Let $M$ be a compact n-dimensional Riemannian
manifold. If
M has nonnegative Ricci curvature everywhere and has positive Ricci curvature at some point,
then M admits a metric with positive Ricci curvature everywhere.
\end{lem}

The following convergence result for Ricci flow in 3-dimension due
to Hamilton is very important in our proof of Theorem \ref{main} for
$n=3$.

\begin{lem}\label{Hamilton}(see Hamilton, \cite{Hamilton})
Let $M$ be a compact 3-manifold which admits a Riemannian metric with strictly positive Ricci curvature.
Then $M$ also admits a metric of constant positive curvature.
\end{lem}

A Riemannian manifold $M$ is said to have \textit{nonnegative
(positive, respectively) isotropic curvature}, if
$$R_{1313}+R_{1414}
+ R_{2323} + R_{2424}- 2 R_{1234}\geq0 (>0, \text{ respectively})$$
for all orthonormal four-frames $\{e_1,e_2,e_3,e_4\}$. This notation
was firstly introduced by Micallef and Moore, where they proved the
following sphere theorem.
\begin{lem}\label{Micallef}(see Micallef and Moore, \cite{Micallef})
Let $M$ be a compact simply connected n-dimensional Riemannian manifolds which has positive isotropic curvature, where $n\ge 4$.
Then $M$ is homeomorphic to a sphere.
\end{lem}

In \cite{Micallef2}, Micallef and Wang  proved that

\begin{lem}\label{Micallef2}(see Micallef and Wang, \cite{Micallef2})
Let $M$ be a closed even-dimensional Riemannian manifold.
If $M$ has positive isotropic curvature, then $b_2(M)=0$.
\end{lem}

Later, H. Seshadri proved that the study of compact manifolds with
nonnegative isotropic curvature reduces to the study of manifolds
with positive isotropic curvature. In view of H. Seshadri's result,
we have
\begin{lem}\label{Lemma2.3}(see H. Seshadri, \cite{Harish})Let $M$ be a compact n-dimensional Riemannian
manifold. If M has nonnegative isotropic curvature everywhere and
has positive isotropic curvature at some point, then M admits a
metric with positive isotropic curvature everywhere.\end{lem}

In order to use the convergence results for the Ricci flow by
Brendle and Schoen (see Theorem \ref{Lemma2.2}) to prove Theorem \ref{main}, we will first prove the following key lemma:
\begin{lem}\label{Lemma2.5}
Let $M$ be an n-dimensional ($n\ge 4$) Lagrangian submanifold in a
complex space form $\bar{M}^n(4c)$ with $c\ge 0$. Suppose that
$$S\le\frac{6n^2H^2}{2n+3}+2c,$$
 then $$R_{1313}+\lambda^2 R_{1414}
+\mu^2 R_{2323} + \lambda^2\mu^2R_{2424}- 2\lambda\mu
R_{1234}\geq0$$ for all orthonormal four-frames
$\{e_1,e_2,e_3,e_4\}$ and all $\lambda,\mu\in[-1,1]$, i.e.
$M\times\mathbb{R}^2$ has nonegative isotropic curvature.
\end{lem}

\begin{proof}
For any orthonormal four-frame $\{e_1,e_2,e_3,e_4\}$, we extend it
to be an orthonormal tangent vector field $\{e_1,\ldots,e_n\}$ and
we get an adapted Lagrangian frame field
$\{e_1,\ldots,e_n,e_{1^*}=Je_1,\ldots,e_{n^*}=Je_n\}$.

 We
denote that
\begin{equation}\label{Hr}
H_r=\frac{1}{n}\sum\limits_{j=1}^nh_{jj}^{r^*},~\forall~ 1\le r\le n.
\end{equation}
By Gauss equation \eqref{24} and \eqref{symm}, we have
\begin{equation}\label{3.0}
\begin{aligned}
R_{1212}=&c+\sum_{r=1}^{n}(h_{11}^{r^*}h_{22}^{r^*}-(h_{12}^{r^*})^2)\\
=&c+\frac{1}{2}(n^2H^2-S)-\frac{1}{2}\sum_{r=1}^n\Big[[\sum_{i=1}^n(h_{ii}^{r^*})^2+\sum_{i\ne j}h_{ii}^{r^*}h_{jj}^{r^*}]
-[\sum_{i=1}^n(h_{ii}^{r^*})^2+\sum_{i\ne j}(h_{ij}^{r^*})^2]\Big]\\
&+\sum_{r=1}^{n}(h_{11}^{r^*}h_{22}^{r^*}-(h_{12}^{r^*})^2)\\
=&c+\frac{1}{2}(n^2H^2-S)-\sum_{r=1}^{n}\Big[\sum_{1\le i<j\le n}(h_{ii}^{r^*}h_{jj}^{r^*}-(h_{ij}^{r^*})^2)-h_{11}^{r^*}h_{22}^{r^*}+(h_{12}^{r^*})^2\Big]\\
=&c+\frac{1}{2}(n^2H^2-S)\\
&-\sum_{r=1}^{n}
\Big[\sum_{j=3}^n(h_{11}^{r^*}+h_{22}^{r^*})h_{jj}^{r^*}+\sum_{3\le
i<j\le n}h_{ii}^{r^*}h_{jj}^{r^*}
-\sum_{j=3}^n(h_{1j}^{r^*})^2-\sum_{2\le i<j\le n}(h_{ij}^{r^*})^2\Big]\\
=&c+\frac{1}{2}(n^2H^2-S)-
\Big[\sum_{r=1}^{n}(\sum_{j=3}^n(h_{11}^{r^*}+h_{22}^{r^*})h_{jj}^{r^*}+\sum_{3\le i<j\le n}h_{ii}^{r^*}h_{jj}^{r^*})\\
&-[\sum_{j=3}^n(h_{1j}^{1^*})^2+\sum_{j=3}^n(h_{1j}^{j^*})^2+\sum_{j=3}^n(h_{1j}^{2^*})^2+2\sum_{3\le i<j\le n}(h_{1j}^{i^*})^2]\\
&-[\sum_{2\le i,j\le n,i\ne j}(h_{jj}^{i^*})^2+\sum_{j\ge
3}(h_{2j}^{1^*})^2+ \sum_{3\le i<j\le n}(h_{ij}^{1^*})^2+3\sum_{2\le
i<j<r\le n}(h_{ij}^{r^*})^2]\Big],
\end{aligned}
\end{equation}
Then by \eqref{symm}, we have
\begin{equation}\label{3.0new}
\begin{aligned}
R_{1212}=&c+\frac{1}{2}(n^2H^2-S)-
\Big[\sum_{r=1}^{n}(\sum_{j=3}^n(h_{11}^{r^*}+h_{22}^{r^*})h_{jj}^{r^*}+\sum_{3\le i<j\le n}h_{ii}^{r^*}h_{jj}^{r^*})\\
&-[\sum_{j=3}^n(h_{11}^{j^*})^2+\sum_{j=3}^n(h_{jj}^{1^*})^2+\sum_{j=3}^n(h_{12}^{j^*})^2+2\sum_{3\le i<j\le n}(h_{ij}^{1^*})^2]\\
&-[\sum_{2\le i,j\le n,i\ne j}(h_{jj}^{i^*})^2+\sum_{j\ge
3}(h_{12}^{j^*})^2+
\sum_{3\le i<j\le n}(h_{ij}^{1^*})^2+3\sum_{2\le i<j<r\le n}(h_{ij}^{r^*})^2]\Big]\\
=&c+\frac{1}{2}(n^2H^2-S)+II_1+II_2+2\sum_{j=3}^n(h_{12}^{j^*})^2+3\sum_{3\le
i<j\le n}(h_{ij}^{1^*})^2+3\sum_{2\le i<j<r\le n}(h_{ij}^{r^*})^2,
\end{aligned}
\end{equation}
where
\begin{equation}\label{3.01}
\begin{aligned}
&II_1=-\sum_{r=1}^2\Big[\sum_{j=3}^n(h_{11}^{r^*}+h_{22}^{r^*})h_{jj}^{r^*}+\sum_{3\le i<j\le n}h_{ii}^{r^*}h_{jj}^{r^*}-\sum_{j=3}^n(h_{jj}^{r^*})^2\Big],\\
&II_2=-\sum_{r=3}^n\Big[\sum_{j=3}^n(h_{11}^{r^*}+h_{22}^{r^*})h_{jj}^{r^*}+\sum_{3\le
i<j\le n}h_{ii}^{r^*}h_{jj}^{r^*}-\sum_{j\ne
r}(h_{jj}^{r^*})^2\Big].
\end{aligned}
\end{equation}

After a straightforward computation, we can
rewrite $II_1$ and $II_2$ as following (see also \cite{BDFV}) by use of \eqref{Hr}:
\begin{equation}\label{3.02}
\begin{aligned}
II_1
=&\frac{1}{2(n+1)}\sum_{r=1}^2\Big[\sum_{j=3}^n((h_{11}^{r^*}+h_{22}^{r^*})-3h_{jj}^{r^*})^2+3\sum_{3\le
i<j\le n}(h_{ii}^{r^*}-h_{jj}^{r^*})^2\\
&-(n-2)(h_{11}^{r^*}+\cdots+h_{nn}^{r^*})^2\Big]\\
=&\frac{1}{2(n+1)}\sum_{r=1}^2\Big[\sum_{j=3}^n((h_{11}^{r^*}+h_{22}^{r^*})-3h_{jj}^{r^*})^2+3\sum_{3\le
i<j\le n}(h_{ii}^{r^*}-h_{jj}^{r^*})^2-(n-2)n^2H_r^2\Big]\\
=&\sum_{r=1}^2\Big\{\frac{3}{2(n+1)(2n+3)}n^2H_r^2-\frac{2n-3}{2(2n+3)}n^2H_r^2\\
&+\frac{1}{2(n+1)}
\big[\sum_{j=3}^n((h_{11}^{r^*}+h_{22}^{r^*})-3h_{jj}^{r^*})^2+3\sum_{3\le
i<j\le n}(h_{ii}^{r^*}-h_{jj}^{r^*})^2\big]\Big\},
\end{aligned}
\end{equation}
\begin{equation}\label{3.022}
\begin{aligned}
II_2
=&\sum_{r=3}^n\Big\{-\frac{2n-3}{2(2n+3)}(h_{11}^{r^*}+\cdots+h_{nn}^{r^*})^2+
\frac{1}{2(2n+3)}\big[\sum_{j\geq 3,j\neq r}(2(h_{11}^{r^*}+h_{22}^{r^*})-3h_{jj}^{r^*})^2\\
&+(2n+3)(h_{11}^{r^*}-h_{22}^{r^*})^2+6\sum_{3\le i<j\le n,i\ne r,j\ne r}(h_{ii}^{r^*}-h_{jj}^{r^*})^2\\
&+2\sum_{j\ge 3,j\ne
r}(h_{rr}^{r^*}-3h_{jj}^{r^*})^2+3(h_{rr}^{r^*}-2(h_{11}^{r^*}+h_{22}^{r^*}))^2
\big]\Big\}\\
=&\sum_{r=3}^n\Big\{-\frac{2n-3}{2(2n+3)}n^2H_r^2+
\frac{1}{2(2n+3)}\big[\sum_{j\geq 3,j\neq r}(2(h_{11}^{r^*}+h_{22}^{r^*})-3h_{jj}^{r^*})^2\\
&+(2n+3)(h_{11}^{r^*}-h_{22}^{r^*})^2+6\sum_{3\le i<j\le n,i\ne r,j\ne r}(h_{ii}^{r^*}-h_{jj}^{r^*})^2\\
&+2\sum_{j\ge 3,j\ne
r}(h_{rr}^{r^*}-3h_{jj}^{r^*})^2+3(h_{rr}^{r^*}-2(h_{11}^{r^*}+h_{22}^{r^*}))^2
\big]\Big\}.
\end{aligned}
\end{equation}
Hence, we get
\begin{equation}\label{3.1}
\begin{aligned}
R_{1212}
=&\frac{1}{2}(\frac{6n^2H^2}{2n+3}+2c-S)+\frac{3}{2(n+1)(2n+3)}(n^2H_1^2+n^2H_2^2)\\
&+2\sum_{j=3}^{n}(h_{12}^{j^*})^2 +3\sum_{3\le i<j\le
n}(h_{ij}^{1^*})^2+ 3\sum_{2\le i<j<r\le n}(h_{ij}^{r^*})^2
\\&+\frac{1}{2(n+1)}\sum_{r=1}^2\Big[\sum_{j=3}^n((h_{11}^{r^*}+h_{22}^{r^*})-3h_{jj}^{r^*})^2+3\sum_{3\le i<j\le n}(h_{ii}^{r^*}-h_{jj}^{r^*})^2\Big]\\
&+\frac{1}{2(2n+3)}\sum_{r=3}^n\Big[\sum_{j\geq 3,j\neq r}(2(h_{11}^{r^*}+h_{22}^{r^*})-3h_{jj}^{r^*})^2+(2n+3)(h_{11}^{r^*}-h_{22}^{r^*})^2\\
&+6\sum_{3\le i<j\le n,i\ne r,j\ne r}(h_{ii}^{r^*}-h_{jj}^{r^*})^2+2\sum_{j\ge 3,j\ne r}(h_{rr}^{r^*}-3h_{jj}^{r^*})^2+3(h_{rr}^{r^*}-2(h_{11}^{r^*}+h_{22}^{r^*}))^2
\Big]\\
\geq&
\frac{1}{2}(\frac{6n^2H^2}{2n+3}+2c-S)+2\sum_{j=3}^{n}(h_{12}^{j^*})^2
+3\sum_{3\le i<j\le n}(h_{ij}^{1^*})^2\\&+ 3\sum_{2\le i<j<r\le
n}(h_{ij}^{r^*})^2+\frac{1}{2}\sum_{r=3}^n(h_{11}^{r^*}-h_{22}^{r^*})^2\\
\geq& \frac{1}{2}(\frac{6n^2H^2}{2n+3}+2c-S),
\end{aligned}
\end{equation}
and the equality in the last inequality holds only if
\begin{equation}\label{r1212}
\left\{
\begin{aligned}
&h_{11}^{1^*}=-h_{22}^{1^*}=a_1,~h_{11}^{2^*}=-h_{22}^{2^*}=a_2,\\
&h_{11}^{r^*}=h_{22}^{r^*}=3b_r,~h_{jj}^{r^*}=4b_r,~h_{rr}^{r^*}=12b_r,~j,r\ge
3,~j\ne r,
\end{aligned}
\right.
\end{equation}
for some numbers $a_1,a_2,b_r$, and the other components of the
second fundamental form are $0$.

By using the first Bianchi identity, Gauss equation \eqref{24} and
\eqref{symm}, we have
\begin{equation}\label{32}
\begin{aligned}
&-R_{1234}\\
=&R_{1342}+R_{1423}\\
=&\sum_{r=1}^n(h_{14}^{r^*}h_{23}^{r^*}-h_{12}^{r^*}h_{34}^{r^*}+h_{12}^{r^*}h_{34}^{r^*}-h_{13}^{r^*}h_{24}^{r^*})\\
=&\sum_{r\ge 5}(h_{14}^{r^*}h_{23}^{r^*}-h_{13}^{r^*}h_{24}^{r^*})\\
&+h_{14}^{1^*}h_{23}^{1^*}-h_{12}^{1^*}h_{34}^{1^*}+h_{12}^{1^*}h_{34}^{1^*}-h_{13}^{1^*}h_{24}^{1^*}
+h_{14}^{2^*}h_{23}^{2^*}-h_{12}^{2^*}h_{34}^{2^*}+h_{12}^{2^*}h_{34}^{2^*}-h_{13}^{2^*}h_{24}^{2^*}\\
&+h_{14}^{3^*}h_{23}^{3^*}-h_{12}^{3^*}h_{34}^{3^*}+h_{12}^{3^*}h_{34}^{3^*}-h_{13}^{3^*}h_{24}^{3^*}
+h_{14}^{4^*}h_{23}^{4^*}-h_{12}^{4^*}h_{34}^{4^*}+h_{12}^{4^*}h_{34}^{4^*}-h_{13}^{4^*}h_{24}^{4^*}\\
=&\sum_{r\ge
5}(h_{14}^{r^*}h_{23}^{r^*}-h_{13}^{r^*}h_{24}^{r^*})\\
&+(h_{33}^{2^*}-h_{11}^{2^*})h_{34}^{1^*}+(h_{11}^{4^*}-h_{33}^{4^*})h_{12}^{3^*}
+(h_{11}^{2^*}-h_{44}^{2^*})h_{34}^{1^*}+(h_{44}^{3^*}-h_{11}^{3^*})h_{12}^{4^*}\\
&+(h_{22}^{1^*}-h_{33}^{1})h_{23}^{4^*}+(h_{33}^{4}-h_{22}^{4^*})h_{23}^{1^*}
+(h_{44}^{1^*}-h_{22}^{1^*})h_{23}^{4^*}+(h_{22}^{3^*}-h_{44}^{3^*})h_{12}^{4^*},
\end{aligned}
\end{equation}
which together with \eqref{3.1} immediately give that
\begin{equation}\label{3.2}
\begin{aligned}
&R_{1313}+\lambda^2R_{1414}+\mu^2R_{2323}+\lambda^2\mu^2R_{2424}-2\lambda\mu
R_{1234}\\
\geq&\frac{1+\lambda^2+\mu^2+\lambda^2\mu^2}{2}(\frac{6n^2H^2}{2n+3}+2c-S)\\
&+2\sum_{j\ne 1,3}(h_{13}^{j^*})^2+3\sum_{i,j\ne
1,3,i<j}(h_{ij}^{1^*})^2
+3\sum_{i<j<r,i\ne 1}(h_{ij}^{r^*})^2+\frac{1}{2}\sum_{r\ne 1,3}(h_{11}^{r^*}-h_{33}^{r^*})^2\\
&+\lambda^2(2\sum_{j\ne 1,4}(h_{14}^{j^*})^2+3\sum_{i,j\ne
1,4,i<j}(h_{ij}^{1^*})^2
+3\sum_{i<j<r,i\ne 1}(h_{ij}^{r^*})^2+\frac{1}{2}\sum_{r\ne 1,4}(h_{11}^{r^*}-h_{44}^{r^*})^2)\\
&+\mu^2(2\sum_{j\ne 2,3}(h_{23}^{j^*})^2+3\sum_{i,j\ne
2,3,i<j}(h_{ij}^{2^*})^2
+3\sum_{i<j<r,i\ne 2,j\ne 2}(h_{ij}^{r^*})^2+\frac{1}{2}\sum_{r\ne 2,3}(h_{22}^{r^*}-h_{33}^{r^*})^2)\\
&+\lambda^2\mu^2(2\sum_{j\ne 2,4}(h_{24}^{j^*})^2+3\sum_{i,j\ne
2,4,i<j}(h_{ij}^{2^*})^2
+3\sum_{i<j<r,i\ne 2,j\ne 2}(h_{ij}^{r^*})^2+\frac{1}{2}\sum_{r\ne 2,4}(h_{22}^{r^*}-h_{44}^{r^*})^2)\\
&+2\lambda\mu\Big[\sum_{r\ge
5}(h_{14}^{r^*}h_{23}^{r^*}-h_{13}^{r^*}h_{24}^{r^*})\\
&+(h_{33}^{2^*}-h_{11}^{2^*})h_{34}^{1^*}+(h_{11}^{4^*}-h_{33}^{4^*})h_{12}^{3^*}
+(h_{11}^{2^*}-h_{44}^{2^*})h_{34}^{1^*}+(h_{44}^{3^*}-h_{11}^{3^*})h_{12}^{4^*}\\
&+(h_{22}^{1^*}-h_{33}^{1})h_{23}^{4^*}+(h_{33}^{4}-h_{22}^{4^*})h_{23}^{1^*}
+(h_{44}^{1^*}-h_{22}^{1^*})h_{23}^{4^*}+(h_{22}^{3^*}-h_{44}^{3^*})h_{12}^{4^*}\Big]\\
=&\frac{1+\lambda^2+\mu^2+\lambda^2\mu^2}{2}(\frac{6n^2H^2}{2n+3}+2c-S)\\
&+2\sum_{j\ne 1,3}(h_{13}^{j^*})^2+2\lambda^2\sum_{j\ne
1,4}(h_{14}^{j^*})^2+2\mu^2\sum_{j\ne
2,3}(h_{23}^{j^*})^2+2\lambda^2\mu^2\sum_{j\ne
2,4}(h_{24}^{j^*})^2\\
&+2\lambda\mu\sum_{r\ge
5}(h_{14}^{r^*}h_{23}^{r^*}-h_{13}^{r^*}h_{24}^{r^*})\\
&+3\sum_{i,j\ne 1,3,i<j}(h_{ij}^{1^*})^2 +3\sum_{i\ne
1,i<j<r}(h_{ij}^{r^*})^2+\frac{1}{2}\lambda^2\mu^2\sum_{r\ne
2,4}(h_{22}^{r^*}-h_{44}^{r^*})^2\\
&+2\lambda\mu\Big[(h_{44}^{1^*}-h_{22}^{1^*})h_{23}^{4^*}+(h_{22}^{3^*}-h_{44}^{3^*})h_{12}^{4^*}\Big]\\
&+\lambda^2[3\sum_{i,j\ne 1,4,i<j}(h_{ij}^{1^*})^2
+3\sum_{i\ne
1,i<j<r}(h_{ij}^{r^*})^2]+\frac{1}{2}\mu^2\sum_{r\ne
2,3}(h_{22}^{r^*}-h_{33}^{r^*})^2\\
&+2\lambda\mu\Big[(h_{22}^{1^*}-h_{33}^{1})h_{23}^{4^*}+(h_{33}^{4}-h_{22}^{4^*})h_{23}^{1^*}
\Big]\\
&+\mu^2[3\sum_{i,j\ne 2,3,i<j}(h_{ij}^{2^*})^2
+3\sum_{i\ne 2,j\ne 2,i<j<r}(h_{ij}^{r^*})^2]
+\frac{1}{2}\lambda^2\sum_{r\ne
1,4}(h_{11}^{r^*}-h_{44}^{r^*})^2\\
&+2\lambda\mu\Big[(h_{11}^{2^*}-h_{44}^{2^*})h_{34}^{1^*}+(h_{44}^{3^*}-h_{11}^{3^*})h_{12}^{4^*}\Big]\\
&+\lambda^2\mu^2[3\sum_{i,j\ne 2,4,i<j}(h_{ij}^{2^*})^2
+3\sum_{i\ne 2,j\ne 2,i<j<r}(h_{ij}^{r^*})^2] +\frac{1}{2}\sum_{r\ne
1,3}(h_{11}^{r^*}-h_{33}^{r^*})^2
\\
&+2\lambda\mu\Big[(h_{33}^{2^*}-h_{11}^{2^*})h_{34}^{1^*}+(h_{11}^{4^*}-h_{33}^{4^*})h_{12}^{3^*}
\Big]\\
\geq&\frac{1+\lambda^2+\mu^2+\lambda^2\mu^2}{2}(\frac{6n^2H^2}{2n+3}+2c-S).
\end{aligned}
\end{equation}

Hence, by the assumption that $S\le\frac{6n^2H^2}{2n+3}+2c$, we get
\begin{equation}\label{3.20}
\begin{aligned}
&R_{1313}+\lambda^2R_{1414}+\mu^2R_{2323}+\lambda^2\mu^2R_{2424}-2\lambda\mu
R_{1234}\geq 0,
\end{aligned}
\end{equation}
 i.e. $M\times\mathbb{R}^2$ has nonnegative isotropic
curvature.
\end{proof}

\noindent\textbf{Proof of Theorem \ref{main}:} We denote by
$\tilde{M}$ the universal cover of $M$, from Lemma \ref{Lemma2.5} we
know that $M\times\mathbb{R}^2$ has nonegative isotropic curvature,
hence $\tilde{M}\times\mathbb{R}^2$ also has nonegative isotropic
curvature. We discuss in two cases: (i) $c>0$ and (ii) $c=0$.

(i) $c>0$.

If $n=3$, for any unit tangent vector $u\in T_pM$ at the point $p\in M$, we can choose an orthonormal three-frame $\{e_1,e_2,e_3\}$ such that
$e_1=u$. From \eqref{3.1} we have
 \begin{equation}\label{3.10}
\begin{aligned}
Ric(u)&=R_{1212}+R_{1313}\\
&\geq \frac{6n^2H^2}{2n+3}+2c-S=6H^2+2c-S,
\end{aligned}
\end{equation}
and the equality holds only if
$R_{1212}=R_{1313}=\frac{1}{2}(6H^2+2c-S)$, then from \eqref{r1212}
we get the equality holds only if
 $h_{ij}^{k^*}=0,1\le i,j,k\le
3$, which means $p$ is a totally geodesic point and hence $S=6H^2$.
Hence we conclude that $M$ has positive Ricci curvature. This
together with Hamilton's theorem (see Lemma \ref{Hamilton}) imply that $M$ is
diffeomorphic to a spherical space form. In particular, if $M$ is
simply connected, then $M$ is diffeomorphic to $\mathbb{S}^3$.

If $n\geq 4$, for any unit vector $u$ at a point $p\in M$, take
$e_1=u$, then we have $Ric(u)=\sum\limits_{k=2}^nR_{1k1k}$, from
\eqref{3.1} we know that $Ric(u)\ge
\frac{n-1}{2}(\frac{6n^2H^2}{2n+3}+2c-S)$, and the equality holds
only if
$R_{1212}=\cdots=R_{1n1n}=\frac{1}{2}(\frac{6n^2H^2}{2n+3}+2c-S)$,
then from \eqref{r1212} we get the equality holds only if
$h_{ij}^{k^*}=0,1\le i,j,k\le n$, which means $p$ is a totally
geodesic point and hence $S=\frac{6n^2H^2}{2n+3}$. Hence we conclude
that $M$ has positive Ricci curvature. Since $M$ is compact and has
positive Ricci curvature, by a theorem of Myers, $\tilde{M}$ is also
compact.

For any orthonormal four-frame $\{e_1,e_2,e_3,e_4\}$ at a point
$p\in T_pM$, let $\lambda=\mu=1$ in \eqref{3.2}, we immediately get
\begin{equation}\label{3.8}R_{1313}+R_{1414}+R_{2323}+R_{2424}-2
R_{1234}\geq 2(\frac{6n^2H^2}{2n+3}+2c-S),\end{equation} and from
\eqref{3.1}, \eqref{r1212} and \eqref{3.2} we know that the equality
in \eqref{3.8} holds only if $h_{ij}^{k^*}=0,1\le i,j,k\le n$, which
means $p$ is a totally geodesic point and hence
$S=\frac{6n^2H^2}{2n+3}$. Hence we conclude that $M$ has positive
isotropic curvature, which implies that $\tilde{M}$ also has
positive isotropic curvature.

We have shown that $\tilde{M}$ is compact and has positive isotropic
curvature, by using a theorem due to Micallef and Moore (see Lemma
\ref{Micallef}), we get $\tilde{M}$ is homeomorphic to
$\mathbb{S}^n$, hence $\tilde{M}$ is locally irreducible and the
locally symmetric metric of $\tilde{M}$ would have to be of constant
positive sectional curvature (see also Remark (ii) of \cite{Harish}
and Lemma 2.4 of \cite{Xu-Gu2}). As $\tilde{M}$ has positive
isotropic curvature, from a theorem of Micallef and Wang (see Lemma
\ref{Micallef2}), $\tilde{M}$ can not be a K\"{a}hler manifold.
Since $\tilde{M}\times\mathbb{R}^2$ has nonegative isotropic
curvature, by combining Theorem \ref{Lemma2.2}, we conclude that
$\tilde{M}$ is diffeomorphic to $\mathbb{S}^n$, which implies that
$M$ is diffeomorphic to a spherical space form. In particular, if
$M$ is simply connected, then  $M$ is diffeomorphic to
$\mathbb{S}^n$.

(ii) $c=0$.

If $n=3$, for any unit tangent vector $u\in T_pM$ at the point $p\in
M$, we can choose an orthonormal three-frame $\{e_1,e_2,e_3\}$ such
that $e_1=u$. From \eqref{3.1} we have
 \begin{equation}
\begin{aligned}
Ric(u)&=R_{1212}+R_{1313}\\
&\geq \frac{6n^2H^2}{2n+3}-S=6H^2-S,
\end{aligned}
\end{equation}
and the equality holds only if
$R_{1212}=R_{1313}=\frac{1}{2}(\frac{6n^2H^2}{2n+3}+2c-S)$, then
from \eqref{r1212} we get the equality holds only if
 $h_{ij}^{k^*}=0,1\le i,j,k\le
3$, which means $p$ is a totally geodesic point and hence $S=6H^2$.
Hence we get $Ric(u)\geq 0$ and $Ric(u)=0$ can only happen at the
totally geodesic points. Since $M$ is a compact submanifold in
$\mathbb{C}^n$, $M$ can not be minimal and hence can not be totally
geodesic, which implies that there exists a point $p\in M$ such that
$Ric(u)>0$ for any unit tangent vector $u\in T_pM$, i.e. we get the
Ricci curvature of $M$ is quasi-positive, then by Aubin's theorem
(see Lemma \ref{Aubin}) we have $M$ admits a metric with positive Ricci
curvature. This together with Hamilton's theorem (see Lemma \ref{Hamilton})
imply that $M$ is diffeomorphic to a spherical space form. In
particular, if $M$ is simply connected, then $M$ is diffeomorphic to
$\mathbb{S}^3$.

If $n\ge 4$, after a same argument with the case for $n=3$, we get
the Ricci curvature of $M$ is quasi-positive, then by Aubin's
theorem (see Lemma \ref{Aubin}) we have $M$ admits a metric with
positive Ricci curvature. Since $M$ is compact and has positive
Ricci curvature, by a theorem of Myers, $\tilde{M}$ is also compact.

For any orthonormal four-frame $\{e_1,e_2,e_3,e_4\}$ at a point
$p\in T_pM$, let $\lambda=\mu=1$ in \eqref{3.2}, we immediately get
\begin{equation}\label{3.14}R_{1313}+R_{1414}+R_{2323}+R_{2424}-2
R_{1234}\geq 2(\frac{6n^2H^2}{2n+3}-S),\end{equation} and from
\eqref{3.1}, \eqref{r1212} and \eqref{3.2} we know that the equality
in \eqref{3.14} holds only if $h_{ij}^{k^*}=0,1\le i,j,k\le n$,
which means $p$ is a totally geodesic point and hence
$S=\frac{6n^2H^2}{2n+3}$. We conclude that $M$ has nonnegative
isotropic curvature and has positive isotropic curvature for some
point in $M$, which together with Lemma \ref{Lemma2.3} imply that
$M$ admits a metric with positive isotropic curvature. Therefore,
$\tilde{M}$ also admits a metric with positive isotropic curvature.

We have shown that $\tilde{M}$ is compact and has admits a metric
with positive isotropic curvature, by using a theorem due to
Micallef and Moore (see Lemma \ref{Micallef}), we get $\tilde{M}$ is
homeomorphic to $\mathbb{S}^n$, hence $\tilde{M}$ is locally
irreducible and the locally symmetric metric of $\tilde{M}$ would
have to be of constant positive sectional curvature (see also Remark
(ii) of \cite{Harish} and Lemma 2.4 of \cite{Xu-Gu2}). As
$\tilde{M}$ admits a metric with positive isotropic curvature, from
a theorem of Micallef and Wang (see Lemma \ref{Micallef2}),
$\tilde{M}$ can not be a K\"{a}hler manifold. Since
$\tilde{M}\times\mathbb{R}^2$ has nonegative isotropic curvature, by
combining Theorem \ref{Lemma2.2}, we conclude that $\tilde{M}$ is
diffeomorphic to $\mathbb{S}^n$, which implies that $M$ is
diffeomorphic to a spherical space form. In particular, if $M$ is
simply connected, then  $M$ is diffeomorphic to $\mathbb{S}^n$.

This completes the proof of Theorem \ref{main}.

\section{A Differentiable Sphere Theorem
for compact Lagrangian submanifolds in a  K\"{a}hler manifold}
In this section, we extend Theorem \ref{main} to compact Lagrangian submanifolds in a K\"{a}hler manifold. We have the following theorem:
\begin{thm}\label{main2}
Let $M$ be an $n(\geq 3$)-dimensional compact Lagrangian submanifold
in a K\"{a}hler manifold $\bar{M}^n$. We denote by $S$ the norm
square of the second fundamental form and $H$ the mean curvature.
Let $\bar{K}(u\wedge v)$ denote the sectional curvature of the
2-dimensional subspace of $T_p\bar{M}^n $ spanned by $u$ and $v$.
Assume that
\begin{equation}\label{pinchingnew}
S\le\frac{3n^2H^2}{n+\frac{3}{2}}+\frac{2}{3}(4\delta-\Delta),
\end{equation}
where $\Delta=\max_{u,v\in T_p\bar{M}^n,\langle
u,Jv\rangle=0}\bar{K}(u\wedge v),~\delta=\min_{u,v\in
T_p\bar{M}^n,\langle u,Jv\rangle=0}\bar{K}(u\wedge v)$, and if
$4\delta-\Delta\equiv 0$ on $\bar{M}^n$, we assume moreover that $M$
is not totally geodesic. Then we have $M$ is diffeomorphic to a
spherical space form. In particular, if $M$ is simply connected,
then $M$ is diffeomorphic to $\mathbb{S}^n$.
\end{thm}

\begin{proof}
Let $M$ be a Lagrangian submanfold in a K\"{a}hler manifold
$\bar{M}^n$, if we denote the Levi-Civita connections on $M$,
$\bar{M}^n$ and the normal bundle by $\nabla$, $D$ and
$\nabla_X^{\bot}$, respectively, then \eqref{21}-\eqref{symm} still
hold.

If we denote the components of
curvature tensor of $\nabla$ and $D$ by $R_{ijkl}$ and $\bar{R}_{ijkl}$, respectively, then the equations of
Gauss
 are given by
\begin{equation}\label{24new}
R_{ijkl}=\bar{R}_{ijkl}
+\sum_{r=1}^n(h_{ik}^{r^*}h_{jl}^{r^*}-h_{il}^{r^*}h_{jk}^{r^*}).
\end{equation}

By a similar method of  proving Berger's inequality (see
\cite{Berger}, or see \cite{Andrews2}, Lemma 2.50), for orthonormal frames
$\{e_i,e_j,e_k,e_l\}$ which are orthogonal to
$\{Je_i,Je_j,Je_k,Je_k\}$, we have
\begin{equation}\label{Berger}
|\bar{R}_{ijkl}|\le \frac{2}{3}(\Delta-\delta).
\end{equation}

By using \eqref{24new} and \eqref{Berger} together instead of
\eqref{24}, Theorem \ref{main2} can be proved after a argument
analogous to that in the proof of Theorem \ref{main}, we omit the
details here.
\end{proof}

\begin{rem}
We note that Theorem \ref{main2} is a generalization of Theorem
\ref{main}. In fact, in Theorem \ref{main2}, if we take $\bar{M}^n$
to be a complex space form $\bar{M}^n(4c)$ with $c\ge 0$, then we
immediately get Theorem \ref{main}.
\end{rem}

\bigskip
\noindent {\bf Acknowledgements}: The authors would like to express
their thanks to Professor Ben Andrews for his helpful discussions
and valuable suggestions.

\begin{flushleft}
\medskip\noindent

Haizhong Li: {\sc Department of Mathematical Sciences, Tsinghua
University, Beijing 100084, People's Republic of China.} \\ E-mail:
hli@math.tsinghua.edu.cn

Xianfeng Wang: {\sc School of Mathematical Sciences, Nankai
University, Tianjin 300071, People's Republic of China.} \\E-mail:
wangxianfeng@nankai.edu.cn

\end{flushleft}

\end{document}